\documentclass[12pt]{scrartcl}

\usepackage[]{amsmath, amssymb,amsfonts,amsthm,mathtools,geometry}
\usepackage{braket}
\usepackage{enumerate}
\usepackage{color}
\usepackage{tikz-cd}
\usepackage{tikz}
\tikzset{
v/.style={
  circle, draw, inner sep=2pt, minimum size=6pt, fill=white}
}

\theoremstyle{plain}% default
\newtheorem{theorem}{Theorem}[section]
\newtheorem{lemma}[theorem]{Lemma}
\newtheorem{proposition}[theorem]{Proposition}
\newtheorem{corollary}[theorem]{Corollary}

\theoremstyle{definition}

\newtheorem{remark}[theorem]{Remark}

\DeclareMathOperator{\Sym}{Sym}
\DeclareMathOperator{\Hom}{Hom}
\DeclareMathOperator{\type}{type}
\DeclareMathOperator{\St}{St}
\DeclareMathOperator{\Cograph}{\mathbf{Cograph}}
\title {The Chromatic Symmetric Functions of Trivially Perfect Graphs and Cographs}
\author{ 
Shuhei Tsujie\thanks{Department of Mathematics, Hokkaido University, North 10, West 8, Kita-ku, Sapporo 060-0810, JAPAN E-mail: tsujie@math.sci.hokudai.ac.jp}
}
\date{}
\begin{document}
\maketitle

\begin{abstract}
Richard P. Stanley defined the chromatic symmetric function of a simple graph and has conjectured that every tree is determined by its chromatic symmetric function. 
Recently, Takahiro Hasebe and the author proved that the order quasisymmetric functions, which are analogs of the chromatic symmetric functions, distinguish rooted trees. 
In this paper, using a similar method, we prove that the chromatic symmetric functions distinguish trivially perfect graphs.  
Moreover, we also prove that claw-free cographs, that is, $ \{K_{1,3},P_{4}\} $-free graphs belong to a known class of $ e $-positive graphs. 
\end{abstract}


{\footnotesize {\textit{2010 MSC:}}
05C15, %Coloring of graphs and hypergraphs
05C25, %Graphs and abstract algebra (groups, rings, ﬁelds, etc.)
05C31, %Graph polynomials
05C60, %Isomorphism problems (reconstruction conjecture, etc.) andhomomorphisms (subgraph embedding, etc.)
05E05, %Symmetric functions and generalizations

% 05Cxx Graph theory
%05Exx Algebraic combinatorics
}

\section{Introduction}\label{Sec:introduction}
Let $ G=(V_{G},E_{G}) $ be a finite simple graph. 
A \textbf{proper coloring} of $ G $ is a function $ \kappa \colon V_{G} \to \mathbb{N} = \{1, 2, \dots\} $ such that $ \{u,v\} \in E_{G} $ implies $ \kappa(u) \neq \kappa(v) $. 
Every proper coloring of $ G $ can be regarded as a graph homomorphism from $ G $ to $ K_{\mathbb{N}} $, the complete graph on $ \mathbb{N} $. 
Let $ \Hom(G,K_{\mathbb{N}}) $ denote the set of proper colorings of $ G $. 
Stanley \cite{stanley1995symmetric-aim} defined the \textbf{chromatic symmetric function} of $ G $ as follows: 
\begin{align*}
X(G,\boldsymbol{x}) \coloneqq \sum_{\kappa \in \Hom(G,K_{\mathbb{N}})} \prod_{v \in V_{G}}x_{\kappa(v)}, 
\end{align*}
where $ \boldsymbol{x} $ denotes infinitely many indeterminates $ (x_{1}, x_{2}, \dots) $. 
By definition, the chromatic symmetric function is homogeneous of degree $ |V_{G}| $. 

Stanley conjectured in \cite{stanley1995symmetric-aim} that the chromatic symmetric function distinguishes trees. 
Namely, if two trees $ T_{1}, T_{2} $ have the same chromatic symmetric function, then $ T_{1} $ and $ T_{2} $ are isomorphic. 

A finite poset $ P $ admits the order quasisymmetric functions, which are kinds of $ P $-partition generating functions studied by Gessel \cite{gessel1984multipartite-cm}. 
The order quasisymmetric functions are considered to be analogs of the chromatic symmetric function. 
A recent study \cite{hasebe2017order-joac} by Hasebe and the author showed that the order quasisymmetric functions distinguish rooted trees (with the natural poset structures). 
The proof is based on algebraic structures of the ring of quasisymmetric functions. 
In this paper, we will focus on algebraic structures of the ring of symmetric functions and consider the similar problem for trivially perfect graphs. 

We will define classes of graphs which are treated in this paper. 
Let $ G,H $ be simple graphs. 
The \textbf{disjoint union} $ G \sqcup H $ is defined by $ V_{G \sqcup H} \coloneqq V_{G} \sqcup V_{H} $ and $ E_{G\sqcup H} \coloneqq E_{G} \sqcup E_{H} $ (the set theoretical disjoint unions). 
The \textbf{join} $ G + H $ is defined by $ V_{G + H} \coloneqq V_{G} \sqcup V_{H} $ and $ E_{G+H} \coloneqq E_{G} \sqcup E_{H} \sqcup \Set{\{u,v\} | u \in V_{G}, v \in V_{H}} $. 
Note that some authors use the symbol $ ``+" $ for disjoint unions. 
See Figure \ref{Fig:disjoint union and join} for examples. 
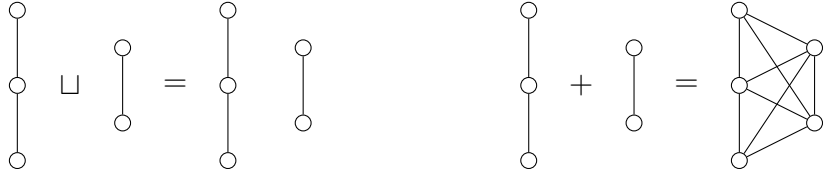
\begin{figure}[t]
\centering
\begin{tikzpicture}
\draw (0,-1) node[v](g1){};
\draw (0, 0) node[v](g2){};
\draw (0, 1) node[v](g3){};
\draw (g1)--(g2)--(g3);
\draw (0.7,0) node(){$ \sqcup $};
\draw (1.4,-0.5) node[v](h1){};
\draw (1.4, 0.5) node[v](h2){};
\draw (h1)--(h2);
\draw (2.1,0) node(){$ = $}; 
\draw (2.8,-1) node[v](gg1){};
\draw (2.8, 0) node[v](gg2){};
\draw (2.8, 1) node[v](gg3){};
\draw (gg1)--(gg2)--(gg3);
\draw (3.8,-0.5) node[v](hh1){};
\draw (3.8, 0.5) node[v](hh2){};
\draw (hh1)--(hh2);
\end{tikzpicture}
\hspace{25mm}
\begin{tikzpicture}
\draw (0,-1) node[v](g1){};
\draw (0, 0) node[v](g2){};
\draw (0, 1) node[v](g3){};
\draw (g1)--(g2)--(g3);
\draw (0.7,0) node(){$ + $};
\draw (1.4,-0.5) node[v](h1){};
\draw (1.4, 0.5) node[v](h2){};
\draw (h1)--(h2);
\draw (2.1,0) node(){$ = $}; 
\draw (2.8,-1) node[v](gg1){};
\draw (2.8, 0) node[v](gg2){};
\draw (2.8, 1) node[v](gg3){};
\draw (gg1)--(gg2)--(gg3);
\draw (3.8,-0.5) node[v](hh1){};
\draw (3.8, 0.5) node[v](hh2){};
\draw (hh1)--(hh2);
\draw (hh1)--(gg1)--(hh2);
\draw (hh1)--(gg2)--(hh2);
\draw (hh1)--(gg3)--(hh2);
\end{tikzpicture}
\caption{Examples of the disjoint union and the join}\label{Fig:disjoint union and join}
\end{figure}

Some classes $ \mathcal{C} $ of simple graphs can be generated by graph operations. 
We consider the following rules. 
\begin{enumerate}[(1)]
\item \label{rule 1} $ K_{1} \in \mathcal{C} $. 
\item \label{rule 2} If $ G \in \mathcal{C} $, then $ G \sqcup K_{1} \in \mathcal{C} $. 
\item \label{rule 3} If $ G \in \mathcal{C} $, then $ G + K_{1} \in \mathcal{C} $. 
\item \label{rule 4} If $ G,H \in \mathcal{C} $, then $ G \sqcup H \in \mathcal{C} $. 
\item \label{rule 5} If $ G,H \in \mathcal{C} $, then $ G+H \in \mathcal{C} $. 
\item \label{rule 6} If $ G \in \mathcal{C} $, then $ \overline{G} \in \mathcal{C} $. 
\end{enumerate}
Note that $ K_{n} $ denotes the complete graph on $ n $ vertices and $ \overline{G} $ denotes the complement of $ G $. 

A member of the class generated by rules (\ref{rule 1},\ref{rule 2},\ref{rule 3}) is called a \textbf{threshold graph}. 
Threshold graphs were introduced by Chv{\'a}tal and Hammer \cite{chvatal1977aggregation-aodm} by a different definition and they gave several characterizations. 
Our definition of threshold graphs is equivalent to the original definition by \cite[Theorem 1]{chvatal1977aggregation-aodm}. 

A member of the class generated by rules (\ref{rule 1},\ref{rule 3},\ref{rule 4}) is called a \textbf{trivially perfect graph} (or a \textbf{quasi-threshold graph}). 
Trivially perfect graphs were introduced by Wolk \cite{wolk1962comparability-potams,wolk1965note-potams} as a comparability graph of an order-theoretic tree. 
A number of characterizations for trivially perfect graphs are known. 
Our definition of trivially perfect graphs is equivalent to the original definition by \cite[Theorem 3]{jing-ho1996quasi-threshold-dam}. 

A member of the class generated by rules (\ref{rule 1},\ref{rule 4},\ref{rule 6}) is called a \textbf{cograph} (short for \textbf{complement reducible graph}). 
Cographs were discovered independently by several researchers and many characterizations are known. 
In the definition, we can replace the rule (\ref{rule 6}) by (\ref{rule 5}) since we have the formula $ G + H = \overline{\overline{G} \sqcup \overline{H}} $.

Obviously, we have the inclusions 
\begin{align*}
\{\text{threshold graphs}\} \subseteq \{\text{trivially perfect graphs}\} \subseteq \{\text{cographs}\}. 
\end{align*}

For a class $ \mathcal{F} $ of simple graphs, a simple graph is said to be \textbf{$ \mathcal{F} $-free} if it has no induced subgraphs isomorphic to a member of $ \mathcal{F} $. 
The three classes above have forbidden induced subgraph characterizations. 
\begin{theorem}[{\cite[Theorem 3]{chvatal1977aggregation-aodm}},\ {\cite[Theorem 2]{golumbic1978trivially-dm}}, \ {\cite[Theorem 2]{corneil1981complement-dam}}]\label{FISC}
Let $ G $ be a simple graph. 
\begin{enumerate}[(1)]
\item \label{FISC1} $ G $ is threshold if and only if $ G $ is $ \{2K_{2},C_{4},P_{4}\} $-free. 
\item \label{FISC2} $ G $ is trivially perfect if and only if $ G $ is $ \{C_{4}, P_{4}\} $-free. 
\item \label{FISC3} $ G $ is a cograph if and only if $ G $ is $ P_{4} $-free. 
\end{enumerate}
Here, $ 2K_{2} = K_{2} \sqcup K_{2} $, $ C_{4} $ is a cycle of length four, and $ P_{4} $ is a path on four vertices (see Figure \ref{Fig:forbidden}). 
\end{theorem}
\begin{figure}[t]
\centering
\begin{tikzpicture}
\draw (0,1) node[v](1){};
\draw (0,0) node[v](2){};
\draw (1,0) node[v](3){};
\draw (1,1) node[v](4){};
\draw (1)--(2);
\draw (3)--(4);
\draw (0.5,-0.5) node(){$ 2K_{2} $};
\end{tikzpicture}
\qquad
\begin{tikzpicture}
\draw (0,1) node[v](1){};
\draw (0,0) node[v](2){};
\draw (1,0) node[v](3){};
\draw (1,1) node[v](4){};
\draw (1)--(2)--(3)--(4)--(1);
\draw (0.5,-0.5) node(){$ C_{4} $};
\end{tikzpicture}
\qquad
\begin{tikzpicture}
\draw (0,1) node[v](1){};
\draw (0,0) node[v](2){};
\draw (1,0) node[v](3){};
\draw (1,1) node[v](4){};
\draw (1)--(2)--(3)--(4);
\draw (0.5,-0.5) node(){$ P_{4} $};
\end{tikzpicture}
\caption{The forbidden graphs}\label{Fig:forbidden}
\end{figure}
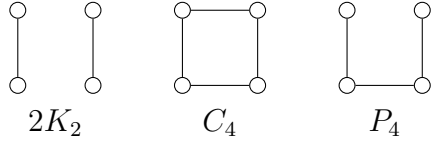

One of two main theorems of this paper is as follows. 
\begin{theorem}\label{main theorem 1}
The chromatic symmetric function distinguishes trivially perfect graphs. 
Namely, if two trivially perfect graphs $ G,H $ have the same chromatic symmetric function, then $ G $ and $ H $ are isomorphic. 
\end{theorem}
\begin{corollary}
The chromatic symmetric function distinguishes threshold graphs. 
\end{corollary}

However, the chromatic symmetric function cannot distinguish cographs. 
We will give the smallest counter example (see Subsection \ref{Subsec:CSF cannot distinguish cographs}). 

To state the other main theorem, we will define $ e $-positivity of graphs. 
An \textbf{integer partition} $ \lambda $ is a finite multiset consisting of positive integers. 
We write an integer partition as $ \langle 1^{r_{1}} \, 2^{r_{2}}, \dots \rangle $, where $ r_{i} $ is the multiplicity of $ i $. 
If $ \lambda \neq \varnothing $ (the empty set), we may write $ \lambda $ as a non-increasing sequence $ (\lambda_{1}, \dotsm \lambda_{\ell}) $ of positive integers. 
We call $ \ell $ the \textbf{length} of $ \lambda $. 

For a positive integer $ k $, we define the elementary symmetric function $ e_{k} $ to be 
\begin{align*}
e_{k} \coloneqq \sum_{i_{1}< \dots < i_{k}}x_{i_{1}} \cdots x_{i_{k}}. 
\end{align*}
Moreover, given an integer partition $ \lambda = (\lambda_{1}, \dots, \lambda_{\ell}) $, define $ e_{\lambda} $ to be 
\begin{align*}
e_{\lambda} \coloneqq e_{\lambda_{1}} \cdots e_{\lambda_{\ell}}
\end{align*}
and $ e_{\varnothing} \coloneqq 1 $. 
It is well known that $ \{e_{\lambda}\}_{\lambda} $ forms a basis for the vector space of symmetric functions over $ \mathbb{Q} $. 
There is another well-known basis $ \{s_{\lambda}\}_{\lambda} $, where $ s_{\lambda} $ denotes the Schur function (we omit the definition in this paper). 

A simple graph is called \textbf{$ e $-positive} (resp. \textbf{$ s $-positive}) if its chromatic symmetric function can be written as non-negative linear combination of elementary symmetric functions (resp. Schur functions). 
It is known that $ e $-positivity implies $ s $-positivity. 

Stanley and Stembridge  (\cite[Conjecture 5.5]{stanley1993immanants-joctsa} and \cite[Conjecture 5.1]{stanley1995symmetric-aim}) have conjectured that the incomparability graph of $ (\boldsymbol{3}+\boldsymbol{1}) $-free poset is $ e $-positive. 
Gasharov \cite[Theorem 2]{gasharov1996incomparability-dm} gave a weaker result: the incomparability graph of $ (\boldsymbol{3}+\boldsymbol{1}) $-free poset is $ s $-positive. 

The \textbf{claw graph} is a complete bipartite graph $ K_{1,3} $ (see Figure \ref{Fig:claw}). 
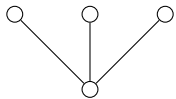
\begin{figure}[t]
\centering
\begin{tikzpicture}
\draw (0,0) node[v](0){};
\draw (-1,1) node[v](1){};
\draw ( 0,1) node[v](2){};
\draw ( 1,1) node[v](3){};
\draw (0)--(1);
\draw (0)--(2);
\draw (0)--(3);
\end{tikzpicture}
\caption{The claw graph $ K_{1,3} $}\label{Fig:claw}
\end{figure}
A $ K_{1,3} $-free graph is called \textbf{claw-free}. 
Note that every incomparability graph of $ (\boldsymbol{3}+\boldsymbol{1}) $-free poset is claw-free. 
Gasharov has conjectured that every claw-free graph is $ s $-positive, which is stated in Stanley's paper \cite[Conjecture 1.4]{stanley1998graph-dm}. 

The complete graph $ K_{n} $ is $ e $-positive since $ X(K_{n},\boldsymbol{x})=n!e_{n} $. 
The edgeless graph $ \overline{K}_{n} $ is also $ e $-positive since $ X(\overline{K}_{n})=e_{1}^{n} $. 
Path graphs and cycle graphs are also known to be $ e $-positive (\cite[Proposition 5.3, Proposition 5.4]{stanley1995symmetric-aim}). 
To prove our second theorem, we need the following lemma. 
\begin{lemma}[{\cite[Excercise 7.47j]{stanley1999enumerative}}]
If the complement of a simple graph $ G $ is $ K_{3} $-free, then $ G $ is $ e $-positive. 
\end{lemma}
Our second main theorem is as follows. 
\begin{theorem}\label{main theorem 2}
Let $ G $ be a claw-free cograph, that is, a $ \{K_{1,3},P_{4}\} $-free graph. 
Then the complement $ \overline{G} $ is $ K_{3} $-free and hence $ G $ is $ e $-positive. 
\end{theorem}

This paper is organized as follows. 
In Section \ref{Sec:preliminaries}, we review a few basic concepts of the ring of symmetric functions and investigate properties of chromatic symmetric functions. 
In Section \ref{Sec:discrimination}, we give a proof of Theorem \ref{main theorem 1} and the counter example for cographs. 
In Section \ref{Sec:e-positivity}, we prove Theorem \ref{main theorem 2}.

\section{Preliminaries}\label{Sec:preliminaries}
\subsection{The ring of symmetric functions}
In this subsection, we review some basic concepts on the theory of symmetric functions. 
Our standard reference is \cite{macdonald1995symmetric}. 
Recall that $ \boldsymbol{x} = (x_{1},x_{2}, \dots) $ denotes infinitely many indeterminates. 
A formal series $ f \in \mathbb{Q}[[\boldsymbol{x}]] $ is called a \textbf{symmetric function} if the following conditions are satisfied. 
\begin{enumerate}[(i)]
\item The degrees of the monomials of $ f $ are bounded. 
\item $ f $ is invariant under any permutation of the indeterminates. 
\end{enumerate}
Let $ \Sym_{\mathbb{Q}} $ denote the subset of the symmetric functions. 
It is well known that $ \Sym_{\mathbb{Q}} $ is a subring of $ \mathbb{Q}[[\boldsymbol{x}]] $, which is called the \textbf{ring of symmetric functions}.

For every integer partition $ \lambda $, we associate it with the \textbf{monomial symmetric function} $ m_{\lambda} $, defined by 
\begin{align*}
m_{\lambda} \coloneqq \sum_{\alpha} \prod_{i=1}^{\infty}x_{i}^{\alpha_{i}},
\end{align*}
where $ \alpha=(\alpha_{1},\alpha_{2}, \dots) $ runs over all distinct rearrangements of $ \lambda $ considered as a sequence $ (\lambda_{1}, \dots, \lambda_{\ell}, 0, \dots) $ of non-negative integers. 
Moreover, we define the \textbf{augmented monomial symmetric function} $ \tilde{m}_{\lambda} $ to be 
\begin{align*}
\tilde{m}_{\lambda} \coloneqq \left(\prod_{i=1}^{\infty}r_{i}!\right)m_{\lambda},  
\end{align*}
where $ r_{i} $ denotes the multiplicity of $ i $ in $ \lambda $, that is $ \lambda=\langle 1^{r_{1}}, 2^{r_{2}}, \dots \rangle $. 
For the empty partition, define $ \tilde{m}_{\varnothing} \coloneqq 1 $. 
It is easy to show that the set $ \{\tilde{m}_{\lambda}\}_{\lambda} $ forms a linear basis for $ \Sym_{\mathbb{Q}} $ over $ \mathbb{Q} $. 

As with the case of symmetric polynomials (in finite indeterminates), the ring of symmetric function $ \Sym_{\mathbb{Q}} $ is a free commutative algebra, that is, there exists a system of symmetric functions $ \{f_{k}\}_{k \in \mathbb{N}} $ which is algebraically independent over $ \mathbb{Q} $ such that $ \Sym_{\mathbb{Q}}=\mathbb{Q}[f_{k} \mid k \in \mathbb{N}] $. 
One of those systems is the system $ \{e_{k}\}_{k \in \mathbb{N}} $ of the elementary symmetric functions. 
Another well-known system is the system
$ \{p_{k}\}_{k \in \mathbb{N}} $ of \textbf{power sum symmetric functions}, defined by 
\begin{align*}
p_{k} \coloneqq \tilde{m}_{k} =  \sum_{i=1}^{\infty}x_{i}^{k}. 
\end{align*}
We also define $ p_{\lambda} \coloneqq p_{\lambda_{1}} \cdots p_{\lambda_{\ell}} $ for an integer partition $ \lambda=(\lambda_{1},\dots,\lambda_{\ell}) $ and $ p_{\varnothing} \coloneqq 1 $. 
Note that the set $ \{p_{\lambda}\}_{\lambda} $ forms a $ \mathbb{Q} $-basis for $ \Sym_{\mathbb{Q}} $.

\subsection{Chromatic symmetric functions}
In this subsection, we review some properties of chromatic symmetric functions and prepare to prove our main theorem. 

For each simple graph $ G $, it is well known that there exists a polynomial $ \chi(G,t) \in \mathbb{Z}[t] $ such that 
\begin{align*}
\chi(G,n) = |\Hom(G,K_{n})| \text{ for all } n \in \mathbb{N}.  
\end{align*}
The polynomial $ \chi(G,t) $ is called the \textbf{chromatic polynomial} of $ G $. 
From the definition of the chromatic symmetric function, we have 
\begin{align*}
X(G,\boldsymbol{1}^{n}) = \chi(G,n) \text{ for all } n \in \mathbb{N}, \text{ where } \boldsymbol{1}^{n} \coloneqq (\underbrace{1, \dots, 1}_{n}, 0, \dots). 
\end{align*}

Recall that every symmetric function is represented by a polynomial in the power sum symmetric functions. 
Define a ring homomorphism $ \varepsilon_{p} \colon \Sym_{\mathbb{Q}} \to \mathbb{Q}[t] $ by the extension of $ \varepsilon_{p}(p_{k}) \coloneqq t $. 
\begin{proposition}
Given a simple graph $ G $, we have 
\begin{align*}
\varepsilon_{p}(X(G,\boldsymbol{x})) = \chi(G,t). 
\end{align*}
\end{proposition}
\begin{proof}
This follows by $ p_{k}(\boldsymbol{1}^{n})=n $ and the discussion above. 
\end{proof}

Every simple graph $ G $ has a decomposition $ G=G_{1}\sqcup \dots \sqcup G_{s} $ into the connected components. 
The chromatic symmetric function $ X(G,\boldsymbol{x}) $ is determined by the connected components of $ G $. 
\begin{proposition}[{\cite[Proposition 2.3]{stanley1995symmetric-aim}}]\label{Stanley CSF disjoint}
Let $ G,H $ be simple graphs. 
Then 
\begin{align*}
X(G \sqcup H, \boldsymbol{x}) = X(G, \boldsymbol{x})X(H, \boldsymbol{x}). 
\end{align*}
\end{proposition}

Cho and van Willigenburg made generators of $ \Sym_{\mathbb{Q}} $ consisting of chromatic symmetric functions. 
\begin{theorem}[{\cite[Theorem 5]{cho2016chromatic-tejoc}}]\label{CvW}
Let $ \{G_{k}\}_{k \in \mathbb{N}} $ be a set of connected simple graphs $ G_{k} $ on $ k $ vertices. 
Then $ \Sym_{\mathbb{Q}}=\mathbb{Q}[X(G_{k},\boldsymbol{x}) \mid k \in \mathbb{N}] $ and $ \{X(G_{k},\boldsymbol{x})\}_{k \in \mathbb{N}} $ is algebraically independent over $ \mathbb{Q} $. 
\end{theorem}

In this paper, the following corollary is required. 
\begin{corollary}\label{CvW cor}
Let $ G $ be a simple graph. 
Then $ G $ is connected if and only if $ X(G,\boldsymbol{x}) $ is irreducible in $ \Sym_{\mathbb{Q}} $. 
\end{corollary}
\begin{proof}
If $ X(G,\boldsymbol{x}) $ is irreducible, then $ G $ is connected by Proposition \ref{Stanley CSF disjoint}. 
To show the converse, suppose that $ G $ is a connected graph on $ n $ vertices. 
Define a collection of graphs $ \{G_{k}\}_{k \in \mathbb{N}} $ by $ G_{n} \coloneqq G $ and $ G_{k} \coloneqq K_{k} $ for any $ k \neq n $. 
By Theorem \ref{CvW}, the set $ \{G_{k}\}_{k \in \mathbb{N}} $ is algebraically independent over $ \mathbb{Q} $ and generates $ \Sym_{\mathbb{Q}} $. 
Assume that $ X(G,\boldsymbol{x}) $ is reducible. 
Then $ X(G,\boldsymbol{x}) $ can be represented as a polynomial in $ \{X(G_{k},\boldsymbol{x})\}_{k < n} $, which is a contradiction. 
Therefore $ X(G,\boldsymbol{x}) $ is irreducible. 
\end{proof}

A set partition of the vertex set $ V_{G} $ of a simple graph $ G $ is a collection $ \pi = \{B_{1}, \dots, B_{\ell}\} $ of non-empty subsets of $ V_{G} $ such that $ B_{1} \sqcup \dots \sqcup B_{s} = V_{G} $. 
Every $ B_{i} $ is called a block. 
the \textbf{type} of a partition $ \pi $ is the integer partition $ \{|B_{1}|, \dots, |B_{\ell}|\} $, denoted by $ \type(\pi) $. 
A set partition is called \textbf{stable} if every block induces an edgeless subgraph of $ G $. 
Let $ \St_{\lambda}(G) $ denote the set of stable partitions of $ G $ of type an integer partition $ \lambda $. 
The chromatic symmetric function can be represented in terms of stable partitions. 

\begin{proposition}[{\cite[Proposition 2.4]{stanley1995symmetric-aim}}]\label{Stanley CSF stable partition}
Given a simple graph $ G $, we have 
\begin{align*}
X(G,\boldsymbol{x}) = \sum_{\lambda} |\St_{\lambda}(G)| \tilde{m}_{\lambda}, 
\end{align*}
where $ \lambda $ runs over all integer partitions. 
\end{proposition}

This proposition may be considered as a generalization of the following proposition. 

\begin{proposition}[{\cite[Theorem 15]{read1968introduction-joct}}]\label{Read coeff falling factorial}
Given a simple graph $ G $, we have 
\begin{align*}
\chi(G,t) = \sum_{\ell=1}^{|V_{G}|}|\St_{\ell}(G)|(t)_{\ell}, 
\end{align*}
where $ \St_{\ell}(G) $ denotes the set of stable partitions of $ G $ consisting of $ \ell $ blocks and $ (t)_{\ell} \in \mathbb{Q}[t] $ denotes the falling factorial. 
Namely $ (t)_{\ell} \coloneqq t(t-1)\cdots(t-\ell+1) $. 
\end{proposition}

Define a map $ \varepsilon_{\tilde{m}} \colon \Sym_{\mathbb{Q}} \to \mathbb{Q}[t] $ by the linear extension of $ \varepsilon_{\tilde{m}}(\tilde{m}_{\lambda}) \coloneqq (t)_{\ell} $, where $ \ell $ is the length of $ \lambda $. 
\begin{proposition}\label{map m_tilde}
Given a simple graph $ G $, we have 
\begin{align*}
\varepsilon_{\tilde{m}}(X(G,\boldsymbol{x})) = \chi(G,t). 
\end{align*}
\end{proposition}
\begin{proof}
This follows immediately by  Propositions \ref{Stanley CSF stable partition} and \ref{Read coeff falling factorial}. 
\end{proof}

Note that the maps $ \varepsilon_{p} $ and $ \varepsilon_{\tilde{m}} $ are different since $ \varepsilon_{\tilde{m}} $ is not a ring homomorphism from $ \Sym_{\mathbb{Q}} $ to $ \mathbb{Q}[t] $. 
However, if we restricts the domain to the set of chromatic symmetric functions, then $ \varepsilon_{p} $ and $ \varepsilon_{\tilde{m}} $ coincide. 

We will introduce multiplications on $ \Sym_{\mathbb{Q}} $ and $ \mathbb{Q}[t] $ such that the map $ \varepsilon_{\tilde{m}} $ becomes a ring homomorphism. 
For integer partitions $ \lambda $ and $ \mu $, let $ \lambda \uplus \mu $ denote the union as multisets. 
For example, $ (3,2,2,1) \uplus (4,2,1) = (4,3,2,2,2,1) $. 
Define a multiplication $ \odot $ on $ \Sym_{\mathbb{Q}} $ by the linear extension of $ \tilde{m}_{\lambda} \odot \tilde{m}_{\mu} \coloneqq \tilde{m}_{\lambda \uplus \mu} $. 
Let $ (\Sym_{\mathbb{Q}}, \odot) $ denote the $ \mathbb{Q} $-algebra equipped with the usual addition and the multiplication $ \odot $. 
Since $ \{\tilde{m}_{\lambda}\}_{\lambda} $ is a $ \mathbb{Q} $-basis for $ \Sym_{\mathbb{Q}} $, the algebra $ (\Sym_{\mathbb{Q}}, \odot) $ is a free commutative algebra generated by $ \{\tilde{m}_{k}\}_{k \in \mathbb{N}} $. 
Moreover, define a multiplication $ \odot $ on $ \mathbb{Q}[t] $ by the linear extension of $ (t)_{\ell} \odot (t)_{m} \coloneqq (t)_{\ell+m} $. 
Let $ (\mathbb{Q}[t],\odot) $ be the $ \mathbb{Q} $-algebra equipped with the usual addition and the multiplication $ \odot $. 
Then $ (\mathbb{Q}[t],\odot) $ is a free commutative algebra generated by $ (t)_{1} $. 
It is easy to verify that the map $ \varepsilon_{\tilde{m}} $ is a ring homomorphism from $ (\Sym_{\mathbb{Q}}, \odot) $ to $ (\mathbb{Q}[t], \odot) $. 

We will see that the chromatic symmetric function of the join $ G+H $ is a product of the chromatic symmetric functions of $ G $ and $ H $ with respect to the multiplication $ \odot $. 
The following proposition is required, which is an analogy of \cite[Proposition 3.11]{hasebe2017order-joac}. 
\begin{proposition}\label{stable partitions of join}
Let $ G $ and $ H $ be simple graphs. 
For every integer partition $ \lambda $, there exists a bijection 
\begin{align*}
\St_{\lambda}(G+H) \simeq \bigsqcup_{\mu \uplus \nu = \lambda} \left(\St_{\mu}(G) \times \St_{\nu}(H)\right). 
\end{align*}
\end{proposition}
\begin{proof}
Every block of a stable partition $ \pi \in \St_{\lambda}(G+H) $ consists of either vertices in $ G $ or vertices in $ H $ since each vertex of $ G $ is adjacent to the vertices of $ H $. 
Let $ \pi_{G}, \pi_{H} $ denote the collection of blocks consisting of vertices in $ G, H $, respectively. 
Then we have that $ \pi = \pi_{G} \sqcup \pi_{H} $. 
Hence the mapping $ \pi \mapsto (\pi_{G},\pi_{H}) $ is a desired bijection. 
\end{proof}
The following proposition is an analogy of \cite[Proposition 3.12]{hasebe2017order-joac}. 
\begin{lemma}\label{CSF join}
Let $ G $ and $ H $ be simple graphs. 
Then 
\begin{align*}
X(G+H,\boldsymbol{x}) = X(G,\boldsymbol{x}) \odot X(H,\boldsymbol{x}). 
\end{align*}
\end{lemma}
\begin{proof}
By Propositions \ref{Stanley CSF stable partition} and \ref{stable partitions of join}, we have 
\begin{align*}
X(G+H,\boldsymbol{x}) &= \sum_{\lambda} |\St_{\lambda}(G+H)| \tilde{m}_{\lambda} \\
&= \sum_{\lambda} \sum_{\mu \uplus \nu = \lambda}|\St_{\mu}(G)||\St_{\nu}(H)|\tilde{m}_{\mu \uplus \nu} \\
&= \sum_{\mu,\nu} |\St_{\mu}(G)||\St_{\nu}(H)|\tilde{m}_{\mu}  \odot \tilde{m}_{\nu} \\
&= \left(\sum_{\mu}|\St_{\mu}(G)|\tilde{m}_{\mu}\right)\odot\left(\sum_{\nu}|\St_{\nu}(H)|\tilde{m}_{\nu}\right) \\
&= X(G,\boldsymbol{x})\odot X(H,\boldsymbol{x}). 
\end{align*}
\end{proof}

Using Proposition \ref{map m_tilde} and Lemma \ref{CSF join}, we can recover the following result of Read. 
\begin{proposition}[{\cite[Theorem 4]{read1968introduction-joct}}]
Let $ G,H $ be simple graphs. 
Then
\begin{align*}
\chi(G + H, t) = \chi(G,t) \odot \chi(H,t). 
\end{align*}
\end{proposition}

\begin{remark}
There is no unary operation on $ \Sym_{\mathbb{Q}} $ which is compatible with taking the complement. 
Stanley's example shows that the graphs $ G $ and $ H $ in Figure \ref{Fig:stanley's example} have the same chromatic symmetric function: 
\begin{align*}
X(G,\boldsymbol{x}) = X(H,\boldsymbol{x}) = \tilde{m}_{11111}+4\tilde{m}_{2111}+2\tilde{m}_{221}. 
\end{align*} 
However, the chromatic symmetric functions of their complements are distinct: 
\begin{align*}
X(\overline{G},\boldsymbol{x}) &= \tilde{m}_{11111}+6\tilde{m}_{2111}+5\tilde{m}_{221}+2\tilde{m}_{311}+2\tilde{m}_{32}, \\
X(\overline{H},\boldsymbol{x}) &= \tilde{m}_{11111}+6\tilde{m}_{2111}+5\tilde{m}_{221}+2\tilde{m}_{311}+\tilde{m}_{32}. 
\end{align*}
\end{remark}
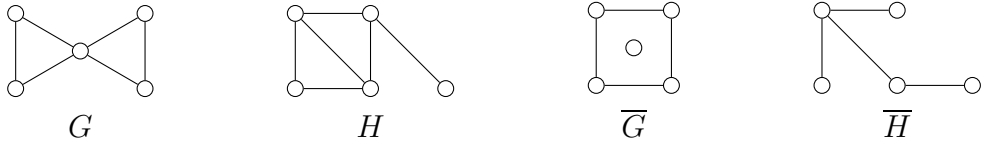
\begin{figure}[t]
\centering
\begin{tikzpicture}
\draw (-0.86, 0.5) node[v](1){};
\draw (-0.86,-0.5) node[v](2){};
\draw (    0,   0) node[v](3){};
\draw ( 0.86, 0.5) node[v](4){};
\draw ( 0.86,-0.5) node[v](5){};
\draw (3)--(1)--(2)--(3)--(4)--(5)--(3);
\draw (0,-1) node(){$ G $};
\end{tikzpicture}
\hspace{15mm}
\begin{tikzpicture}
\draw (0,1) node[v](1){};
\draw (0,0) node[v](2){};
\draw (1,0) node[v](3){};
\draw (1,1) node[v](4){};
\draw (2,0) node[v](5){};
\draw (1)--(2)--(3)--(4)--(1)--(3);
\draw (4)--(5);
\draw (1,-0.5) node(){$ H $};
\end{tikzpicture}
\hspace{15mm}
\begin{tikzpicture}
\draw (0,1) node[v](1){};
\draw (0,0) node[v](2){};
\draw (1,0) node[v](3){};
\draw (1,1) node[v](4){};
\draw (0.5,0.5) node[v](5){};
\draw (1)--(2)--(3)--(4)--(1);
\draw (0.5,-0.5) node(){$ \overline{G} $};
\end{tikzpicture}
\hspace{15mm}
\begin{tikzpicture}
\draw (0,1) node[v](1){};
\draw (0,0) node[v](2){};
\draw (1,0) node[v](3){};
\draw (1,1) node[v](4){};
\draw (2,0) node[v](5){};
\draw (1)--(2);
\draw (1)--(3)--(5);
\draw (1)--(4);
\draw (1,-0.5) node(){$ \overline{H} $};
\end{tikzpicture}
\caption{Stanley's examples and their complements}\label{Fig:stanley's example}
\end{figure}

\section{Discrimination}\label{Sec:discrimination}
\subsection{Discrimination for trivially perfect graphs}
We now ready to prove Theorem \ref{main theorem 1}. 
The following proof is almost as same as the proof of \cite[Theorem 1.3]{hasebe2017order-joac}. 
\begin{proof}[Proof of Theorem \ref{main theorem 1}]
We proceed by induction on $ |V_{G}| $. 
When $ |V_{G}|=1 $, we have $ G=H=K_{1} $. 
Suppose that $ |V_{G}| \geq 2 $. 
Decompose $ G $ and $ H $ into their connected components:
\begin{align*}
G=\bigsqcup_{i=1}^{n}G_{i}, \qquad H=\bigsqcup_{i=1}^{m}H_{i}. 
\end{align*}
By the assumption $ X(G,\boldsymbol{x}) = X(H,\boldsymbol{x}) $ and Proposition \ref{Stanley CSF disjoint}, we have 
\begin{align*}
\prod_{i=1}^{n}X(G_{i},\boldsymbol{x}) =  \prod_{i=1}^{m}X(H_{i},\boldsymbol{x}).
\end{align*}
The ring of symmetric functions $ \Sym_{\mathbb{Q}} $ is a free commutative algebra and hence it is a unique factorization domain. 
Using Corollary \ref{CvW cor}, we have that $ n=m $ and $ X(G_{i},\boldsymbol{x}) = X(H_{i}, \boldsymbol{x}) $ for each $ i $ after a suitable renumbering. 

Assume that $ n \geq 2 $. 
The induced subgraphs $ G_{i},H_{i} $ are also trivially perfect by Theorem \ref{FISC}(\ref{FISC2}) and the number of vertices of $ G_{i} $ is less than $ |V_{G}| $. 
Therefore, by our induction hypothesis, we have that $ G_{i} $ is isomorphic to $ H_{i} $. 
Hence $ G $ and $ H $ are isomorphic. 

Now consider the case $ n=1 $, that is, $ G $ and $ H $ are connected. 
By the definition of trivially perfect graphs, there are trivially perfect graphs $ G^{\prime}, H^{\prime} $ such that $ G=G^{\prime}+K_{1} $ and $ H=H^{\prime}+K_{1} $. 
Since $ X(K_{1},\boldsymbol{x}) = \tilde{m}_{1} $, using Lemma \ref{CSF join}, we have 
\begin{align*}
X(G^{\prime}, \boldsymbol{x}) \odot \tilde{m}_{1} = X(H^{\prime}, \boldsymbol{x}) \odot \tilde{m}_{1}. 
\end{align*}
Since the algebra $ (\Sym_{\mathbb{Q}},\odot) $ is an integral domain, we have $ X(G^{\prime},\boldsymbol{x})=X(H^{\prime},\boldsymbol{x}) $. 
Our induction hypothesis forces that $ G^{\prime} $ is isomorphic to $ H^{\prime} $. 
Thus $ G $ and $ H $ are isomorphic. 
\end{proof}

\subsection{Discrimination for cographs}\label{Subsec:CSF cannot distinguish cographs}
As mentioned in Section \ref{Sec:introduction}, the chromatic symmetric function cannot distinguish cographs. 
We will raise an example. 

A simple graph is called \textbf{coconnected} if its complement is connected. 
Consider a simple graph $ G $ and a decomposition $ \overline{G} = \overline{G}_{1} \sqcup \dots \sqcup \overline{G}_{n} $, where $ \overline{G}_{i} $ is a connected component of $ \overline{G} $. 
Taking complements of the both sides, we obtain $ G = G_{1} + \dots + G_{n} $. 
Every $ G_{i} $ is called a \textbf{coconnected component}. 
Since the connected components of a simple graph are uniquely determined, hence coconnected components are also uniquely determined. 

The isomorphic classes of cographs is closed under taking the disjoint union $ \sqcup $ and taking the join $ + $. 
Let $ \Cograph $ denote the algebraic system equipped with two commutative and associative operations $ \sqcup $ and $ + $ whose underlying set consists of the isomorphic classes of cographs. 
\begin{proposition}\label{cograph free}
The algebraic system $ \Cograph $ is free and generated by $ K_{1} $. 
\end{proposition}
\begin{proof}
Let $ G $ be a cograph. 
We proceed by induction on $ |V_{G}| $. 
If $ |V_{G}|=1 $, then $ G=K_{1} $ and there are no other representations. 
Assume that $ |V_{G}| \geq 2 $. 
By the definition of cographs, $ G $ is either a disjoint union or a join of some cographs. 
By the induction hypothesis, the connected components or the coconnected components of $ G $ are represented uniquely by using $ K_{1} $. 
Therefore $ G $ also has a unique representation by using $ K_{1} $. 
Thus $ \Cograph $ is a free algebraic system. 
\end{proof}
\begin{remark}
One can construct an algebraic system called a commutative De Morgan bisemigroup from $ \Cograph $. 
A generalized result of Proposition \ref{cograph free} was proven by \cite{esik2003free-ac}. 
In \cite{corneil1981complement-dam}, it was shown that every cograph admits a unique cotree representation, which is equivalent to Proposition \ref{cograph free}. 
\end{remark}

For the proof of Theorem \ref{main theorem 1}, it plays an important role that a simple graph is connected if and only if its chromatic symmetric function is irreducible in $ \Sym_{\mathbb{Q}} $ (Corollary \ref{CvW cor}). 
However, there is no reason why the chromatic symmetric function of a coconnected cograph is irreducible in $ (\Sym_{\mathbb{Q}}, \odot) $. 
In fact, we have the following equalities by using Proposition \ref{Stanley CSF stable partition}. 
\begin{align*}
X(K_{2} \sqcup K_{1}, \boldsymbol{x}) &= \tilde{m}_{111}+2\tilde{m}_{21} = \tilde{m}_{1} \odot (\tilde{m}_{11}+2\tilde{m}_{2}), \\
X(K_{6}\sqcup K_{1}, \boldsymbol{x}) &= \tilde{m}_{1111111}+6\tilde{m}_{211111} = \tilde{m}_{11111}\odot(\tilde{m}_{11}+6\tilde{m}_{2}), \\
X(K_{4}\sqcup K_{2},\boldsymbol{x}) &= \tilde{m}_{111111}+8\tilde{m}_{21111}+12\tilde{m}_{2211} = \tilde{m}_{11}\odot(\tilde{m}_{11}+2\tilde{m}_{2})\odot(\tilde{m}_{11}+6\tilde{m}_{2}), \\
X(K_{4},\boldsymbol{x}) &= \tilde{m}_{1111}. 
\end{align*}
By Lemma \ref{CSF join}, these equalities yield that both of the cographs $ (K_{2}\sqcup K_{1})+(K_{6}\sqcup K_{1}) $ and $ (K_{4}\sqcup K_{2})+K_{4} $ have the same chromatic symmetric function 
\begin{align*}
\tilde{m}_{111111}\odot(\tilde{m}_{11}+2\tilde{m}_{2})\odot(\tilde{m}_{1}+6\tilde{m}_{2}). 
\end{align*}
Furthermore, by Proposition \ref{cograph free}, we have that these graphs are not isomorphic (Figure \ref{Fig:example cographs}). 
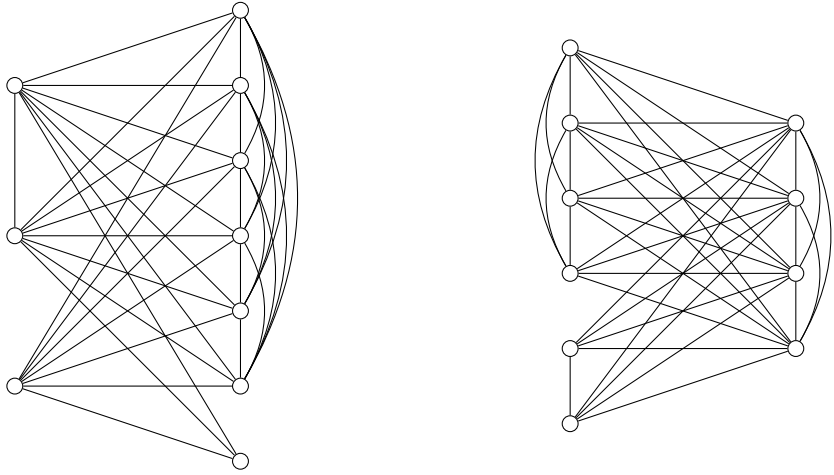
\begin{figure}[t]
\centering
\begin{tikzpicture}[baseline=0]
\draw (0,0) node[v](l3){};
\draw (0,2) node[v](l2){};
\draw (0,4) node[v](l1){};
\draw (l1)--(l2);
\draw (3,-1) node[v](r7){};
\draw (3, 0) node[v](r6){};
\draw (3, 1) node[v](r5){};
\draw (3, 2) node[v](r4){};
\draw (3, 3) node[v](r3){};
\draw (3, 4) node[v](r2){};
\draw (3, 5) node[v](r1){};
\draw (r1)--(r2)--(r3)--(r4)--(r5)--(r6);
\draw[bend left] (r1) to (r3);
\draw[bend left] (r1) to (r4);
\draw[bend left] (r1) to (r5);
\draw[bend left] (r1) to (r6);
\draw[bend left] (r2) to (r4);
\draw[bend left] (r2) to (r5);
\draw[bend left] (r2) to (r6);
\draw[bend left] (r3) to (r5);
\draw[bend left] (r3) to (r6);
\draw[bend left] (r4) to (r6);
\draw (l1)--(r1);
\draw (l1)--(r2);
\draw (l1)--(r3);
\draw (l1)--(r4);
\draw (l1)--(r5);
\draw (l1)--(r6);
\draw (l1)--(r7);
\draw (l2)--(r1);
\draw (l2)--(r2);
\draw (l2)--(r3);
\draw (l2)--(r4);
\draw (l2)--(r5);
\draw (l2)--(r6);
\draw (l2)--(r7);
\draw (l3)--(r1);
\draw (l3)--(r2);
\draw (l3)--(r3);
\draw (l3)--(r4);
\draw (l3)--(r5);
\draw (l3)--(r6);
\draw (l3)--(r7);
\end{tikzpicture}
\hspace{25mm}
\begin{tikzpicture}[baseline=-5mm]
\draw (0,4) node[v](l1){};
\draw (0,3) node[v](l2){};
\draw (0,2) node[v](l3){};
\draw (0,1) node[v](l4){};
\draw (0,0) node[v](l5){};
\draw (0,-1) node[v](l6){};
\draw (l1)--(l2)--(l3)--(l4);
\draw[bend right] (l1) to (l3);
\draw[bend right] (l1) to (l4);
\draw[bend right] (l2) to (l4);
\draw (l5)--(l6);
\draw (3,3) node[v](r1){};
\draw (3,2) node[v](r2){};
\draw (3,1) node[v](r3){};
\draw (3,0) node[v](r4){};
\draw (r1)--(r2)--(r3)--(r4);
\draw[bend left] (r1) to (r3);
\draw[bend left] (r1) to (r4);
\draw[bend left] (r2) to (r4);
\draw (l1)--(r1);
\draw (l1)--(r2);
\draw (l1)--(r3);
\draw (l1)--(r4);
\draw (l2)--(r1);
\draw (l2)--(r2);
\draw (l2)--(r3);
\draw (l2)--(r4);
\draw (l3)--(r1);
\draw (l3)--(r2);
\draw (l3)--(r3);
\draw (l3)--(r4);
\draw (l4)--(r1);
\draw (l4)--(r2);
\draw (l4)--(r3);
\draw (l4)--(r4);
\draw (l5)--(r1);
\draw (l5)--(r2);
\draw (l5)--(r3);
\draw (l5)--(r4);
\draw (l6)--(r1);
\draw (l6)--(r2);
\draw (l6)--(r3);
\draw (l6)--(r4);
\end{tikzpicture}
\caption{The smallest example of two non-isomorphic cographs which have the same chromatic symmetric function}\label{Fig:example cographs}
\end{figure}

\section{$ e $-positivity of claw-free cographs}\label{Sec:e-positivity}
In this section, we will prove Theorem \ref{main theorem 2} and conclude that every claw-free cograph is $ e $-positive. 

\begin{lemma}\label{coconnected components}
Every coconnected component of a connected claw-free cograph is $ K_{1} $ or a disjoint union of two complete graphs. 
\end{lemma}
\begin{proof}
Let $ G $ be a connected claw-free graph. 
If $ G $ is complete, then the assertion holds since $ G $ is the join of some single-vertex graphs. 
Suppose that $ G $ is non-complete. 
The connectivity of $ G $ shows that $ G $ has at least two coconnected components. 
Assume that there is a coconnected component $ G_{1} $ such that it consists of at least three connected components. 
Take vertices $ a,b,c $ from distinct connected components of $ G_{1} $ and take a vertex $ d $ from a coconnected component distinct from $ G_{1} $. 
Then the subgraph of $ G $ induced by $ \{a,b,c,d\} $ is isomorphic to the claw graph, which is a contradiction. 
Therefore the number of connected components of every coconnected component of $ G $ is at most two. 
\end{proof}

Now we ready to prove Theorem \ref{main theorem 2}. 
\begin{proof}[Proof of Theorem \ref{main theorem 2}]
Let $ G $ be a claw-free graph. 
Without loss of generality we may assume that $ G $ is connected and non-complete. 
By Lemma \ref{coconnected components}, our graph $ G $ is one of the following form: 
\begin{align*}
&(G_{1} \sqcup G_{1}^{\prime}) + \dots + (G_{m} \sqcup G_{m}^{\prime}), \\
&(G_{1} \sqcup G_{1}^{\prime}) + \dots + (G_{m} \sqcup G_{m}^{\prime}) + G_{m+1}, 
\end{align*}
where $ G_{i},G_{i}^{\prime} $ are complete graphs on some vertices. 
In order to show that $ \overline{G} $ is $ K_{3} $-free, it suffices to show that any subgraph of $ G $ induced by three vertices $ \{a,b,c\} $ has at least one edge. 

If $ a $ belongs to $ G_{m+1} $, then $ a $ is adjacent to any other vertices. 
In particular, we obtain edges $ \{a,b\} $ and $ \{a,c\} $. 
Suppose that two of $ \{a,b,c\} $ belong to distinct coconnected components. 
Then there is an edge connecting these two vertices. 
Hence we may assume that $ a,b,c $ belong to $ G_{i} \sqcup G_{i}^{\prime} $ for some $ i $. 
In this case, at least two of $ \{a,b,c\} $ belong to the same component and hence we have an edge. 
\end{proof}

\bibliographystyle{amsalpha}
\bibliography{bibfile}

\end{document}